\documentclass[11pt]{article}

\usepackage{amsmath,amssymb,xfrac}
\usepackage{algorithm,algorithmic}
\usepackage{url,hyperref,cleveref}
\usepackage{authblk}
\usepackage{caption,subcaption}
\usepackage[margin=1in]{geometry}
\usepackage{tikz}
\usetikzlibrary{arrows,calc,positioning,shapes.misc}
\usepackage[square,numbers]{natbib}
\usepackage{amsthm}
\newtheorem{theorem}{Theorem}
\newtheorem{proposition}{Proposition}

\newtheorem*{theorem*}{Theorem}
\newtheorem*{proposition*}{Proposition}
\newtheorem*{lemma*}{Lemma}
\newtheorem*{corollary*}{Corollary}
\newtheorem{assumption}{Assumption}
\theoremstyle{definition}
\newtheorem{definition}{Definition}

\newcommand{\mbb}{\mathbb}

\newcommand{\mc}{\mathcal}
\newcommand{\st}{\mathrm{s.t.}\;}

\newcommand{\defn}{\equiv}
\newcommand{\group}[1]{\left( #1 \right)}

\newcommand{\set}[1]{\left\{ #1 \right\}}

\renewcommand{\hat}{\widehat}
\renewcommand{\tilde}{\widetilde}
\renewcommand{\emptyset}{\varnothing}
\DeclareMathOperator{\smallsum}{\textstyle{\sum}}

\newcommand{\allg}{{g}}

\title{Computing normalized Nash equilibria for generalized Nash games with nonconvex players%
\thanks{This preprint has not undergone peer review (when applicable) or any post-submission improvements or corrections.
The Version of Record of this article is published in Annals of Operations Research, and is available online at https://doi.org/10.1007/s10479-025-06968-z.}
}
\date{\today}

\author[1]{Stuart M. Harwood}
\author[1]{Dimitri J. Papageorgiou}

\affil[1]{ExxonMobil Technology and Engineering Company, \protect\\ 1545 Route 22 East, Annandale, NJ 08801 USA}

\begin{document}
\maketitle

\begin{abstract}
Generalized Nash equilibrium (GNE) is a solution concept for complete information games, in which each player's objective function \emph{and} feasible region depend on other players' actions. While numerical methods for finding GNE when players possess convex structure are relatively mature, the same cannot be said when players optimize nonconvex objective functions over nonconvex feasible regions.  Drawing inspiration from the notion of a normalized (or variational) Nash equilibrium, which is a more restrictive class of solutions to generalized Nash games, we extend the ideas of \cite{harwood2024equilibrium} to develop an exact method that can find a normalized Nash equilibrium (NNE) of a problem, when such an NNE exists. By adapting the framework of \cite{harwood2024equilibrium}, we are able to find NNE without any convexity assumptions.  We demonstrate the effectiveness of our method on several nonconvex games. 

\vspace{2mm}
\noindent \textbf{keywords:} bilevel optimization, equilibrium modeling, game theory, nonconvexity, semi-infinite programming. 
\end{abstract}
\section{Introduction}
\label{sec:intro}

Generalized Nash equilibrium (GNE) is a solution concept for complete information games, in which the objective function \emph{and} decision set/space for each player depend on the action of the other players.
Generalized Nash games and GNE have been applied in many areas, including electricity markets \citep{wu2020energy}, energy markets \citep{abada2013generalized,liang2019generalized}, transport \citep{zhou2005generalized}, environment \citep{krawczyk05,wang2019generalized}, and, famously, models of economies \citep{arrowEA54}.
For the most part, numerical methods for finding GNE require the players to have a certain amount of convex structure;
see e.g.~the review \cite{facchinei2010generalized}, as well as \cite{dreves2018select,dreves2019algorithm,dreves2012nonsmooth,vonheusingerEA12}.
Some extensions to specially-structured nonconvex problems have been proposed in \cite{sagratella2019generalized,harks2024generalized}.
In contrast, the recent work by \citet{harwood2024equilibrium} introduced methods for modeling and solving pure Nash equilibrium (PNE) problems that did not not require convexity assumptions.


Extending these new methods to the GNE case presents certain challenges.
Assuming that, for instance, the decision set of each player can be represented as a set of algebraic inequality constraints with parametric dependence on the other players' decisions, one framework for extending methods for PNE to GNE involves penalizing the parametric constraints, thereby reducing the problem in finding a GNE to a problem in finding a PNE.
This characterizes the approaches from \cite{ba2022exact,facchinei2010penalty}, for instance.
Exact penalization of optimization problems typically requires convexity, however, and the work from \cite{ba2022exact,facchinei2010penalty} is no different when applying this to the GNE problem.

Instead, we draw inspiration from the notion of a normalized (or variational) Nash equilibrium, which is a more restrictive class of solutions to generalized Nash games.
In this work, we extend the ideas of \cite{harwood2024equilibrium} to develop a method that can find a normalized Nash equilibrium (NNE) of a problem, when such an NNE exists.
By adapting the framework of \cite{harwood2024equilibrium}, we are able to find NNE without any convexity assumptions.

In order to find a PNE, the central approach of \cite{harwood2024equilibrium} is to instead find a \emph{minimum disequilibrium} point.
While such a point has a couple interesting properties, the main one is that if ``disequilibrium'' is zero, then the point is in fact a PNE.
The problem of minimizing disequilibrium is equivalent to a semi-infinite program, or a mathematical optimization problem with an infinite (and potentially uncountably so) number of constraints.
Fairly mature constraint generation methods have been developed that can solve SIP under mild assumptions
(specifically, without the need for convex structure \citep{blankenshipEA76}).
The challenge of adapting this approach to generalized Nash games is that the problem of finding a minimum disequilibrium becomes a generalized semi-infinite program (GSIP).
Especially in instances without convex structure, solving a GSIP can be quite difficult \citep{guerravazquezEA08}.


The main realization that inspires this work is that in generalized games, NNE are a type of minimum disequilibrium solution that may be found by solving an SIP.
This is fitting, as NNE have had more favorable numerical properties since the introduction of the term in \cite{rosen1965existence}, where a gradient-based procedure was proposed to find NNE of convex games.
From our perspective (that SIP are easier to solve than GSIP), NNE continue to fit this trend.

Normalized Nash equilibrium has been used in favor of GNE in a number of settings, e.g., \cite{dreves2018generalized,ghosh2015normalized,le2022parametrized,wei2015charging},
where NNE is favored for its numerical properties as well as its economic interpretation.
Consequently, the value of NNE has motivated work on numerical methods to find them: 
\citet{dreves2013globalized} propose a globalized Newton method for the computation of normalized Nash equilibria;
\citet{vonheusingerEA12} approach the problem of finding an NNE through a fixed-point formulation.
As noted before, however, these methods require convex structure.



Finally, we reiterate the general setting of this work;
we allow for each player's objective function to be nonconvex, and their decision set to be nonconvex, including the case that certain variables are constrained to be discrete.
Thus, even though the original definition of NNE by \citet{rosen1965existence} required the existence of Lagrange multipliers
(that is, it assumes continuous and more so convex players),
we develop a method that find an NNE in a meaningfully generalized setting.

The rest of this work is organized as follows.
The definition of GNE and its characterization as a solution of a mathematical optimization problem is recalled in \Cref{sec:recap}.
We briefly elaborate on the challenges of solving this optimization problem in \Cref{sec:md_for_gne}.
In \Cref{sec:nne}, we introduce normalized Nash equilibrium and, more generally, minimum normalized disequilibrium as the central notion of solution on which we focus.
Further, we propose an algorithm for finding minimum normalized disequilibrium.
In \Cref{sec:examples}, we demonstrate the effectiveness of this method on some examples.
This includes a computational study of a class of games involving ``knapsack'' players;
we get the somewhat surprising result that these games consistently have equilibria across a range of parameters and data.
We conclude in \Cref{sec:conclusions}.
\section{Definition and characterization of equilibrium}
\label{sec:recap}
The basic definitions and characterization of pure Nash equilibrium from \cite{harwood2024equilibrium} are trivially extended to the generalized case;
see also \cite{harwood2022equilibrium}.
In this section we recall these concepts.

\subsection{Definition of equilibrium}
Consider the (nonempty) set 
\[
	G \subset \mbb{R}^{n_0} \times \mbb{R}^{n_1} \times \dots \times \mbb{R}^{n_m},
\] 
and collection of $m$ player/agent optimization problems parametric in $x \in \mbb{R}^{n_0}$;
for $ i \in \set{1, \dots, m}$, player $i$ is defined by
\begin{equation}
\label{agent_i}
\tag{$\mathcal{A}_i$}
\begin{aligned}
S_i(x) \defn \arg
\min_{y_i}\; & g_i(x,y_i) \\
\st
& (x,y_i) \in F_i
\end{aligned}
\end{equation}
where
$F_i \subset \mbb{R}^{n_0} \times \mbb{R}^{n_i}$
and
$g_i : F_i \to \mbb{R}$.
Also important is the optimal value function of problem~\eqref{agent_i}
\begin{equation}
\label{g_star}
g^*_i(x) \defn \inf_{y_i} \set{ g_i(x,y_i) : (x,y_i) \in F_i }.
\end{equation}
As usual, for any $i$, define 
$g^*_i(x) = +\infty$ if optimization problem~\eqref{agent_i} is infeasible, and 
$g^*_i(x) = -\infty$ if the problem is unbounded.

We have the following definition of a generalized Nash equilibrium (GNE).
Although this definition and overall form of the game is different from previous definitions \citep{facchinei2010generalized}, the two formulations have equivalent modeling power;
see \cite[\S2.3]{harwood2022equilibrium}, specifically Propositions~1 and 2.

\begin{definition}
\label{defn:gne}
A point $(x^*,y_1^*,\dots,y_m^*)$ is a GNE of the collection of \eqref{agent_i} with respect to $G$ if
$(x^*,y_1^*,\dots,y_m^*) \in G$ and $y_i^* \in S_i(x^*)$ for each $i$.
\end{definition}

\noindent
\textbf{Notation.} We will sometimes write
\[
	(y_1, \dots ,y_m) = y \in \mbb{R}^{\group{\sum_i n_i}}.
\]
More generally, a symbol without a subscript refers to a tuple/block vector of the subscripted objects;
for example we have $z = (z_1, z_2, \dots, z_J)$ for some vectors or scalars $z_j$. 
\subsection{Characterization of equilibrium and minimum disequilibrium}
The following result characterizes equilibrium as the solution of an optimization problem.
The objective can be thought of as a measure of ``disequilibrium,'' or roughly the dissatisfaction of all the players in aggregate.
This depends on an $\bar{\mbb{R}}$-valued function $\mu$;
a convenient choice is%
\footnote{We denote the extended-reals by $\bar{\mbb{R}} \defn \mbb{R}\cup\set{-\infty, +\infty}$ and the subset of non-negative extended reals by $\bar{\mbb{R}}_+$.
Extended-real arithmetic proceeds by the usual rules;
we might worry about what to do in the undefined case $\infty - \infty$.
However, this cannot occur in the feasible set of \eqref{md}:
If $(x,y)$ is feasible in \eqref{md}, then for each $i$, $(x, y_i) \in F_i$, and thus player $i$ is feasible and so $g_i^*(x) < +\infty$.
Since $g_i$ is real-valued by assumption, each term $g_i(x,y) - g_i^*(x)$ is strictly greater than $-\infty$, and so an expression like
$\sum_i (g_i(x,y) - g_i^*(x))$ is either finite or $+\infty$.
}
\[
	\mu : (v_1,\dots,v_m) \mapsto \smallsum_{i=1}^m v_i.
\]

\begin{proposition}{\citep[Prop.~3]{harwood2022equilibrium}}
\label{prop:equilibrium}
Let $\mu : \bar{\mbb{R}}^m \to \bar{\mbb{R}}$ be any function satisfying:
\begin{enumerate}\itemsep0pt \parskip0pt
\item if $v \in \bar{\mbb{R}}^m_+$ then $\mu(v) \ge 0$;
\item $v \in \bar{\mbb{R}}^m_+$ and $\mu(v) = 0$ if and only if $v_i = 0$ for all $i$.
\end{enumerate}
Consider
\begin{equation}
\label{md}
\tag{$\mathcal{MD}$}
\begin{aligned}
\delta \defn
\inf_{x,y_1,\dots,y_m}\;
&\mu(\allg(x,y) - \allg^*(x)) \\
\st 
&(x,y_1,\dots,y_m) \in G\\
&(x,y_i) \in F_i, \quad \forall i.
\end{aligned}
\end{equation}
Any $(x^*,y_1^*,\dots, y_m^*)$ is a solution of \eqref{md} with $\delta = 0$ if and only it is a GNE.
\end{proposition}

We note that by the assumptions on $\mu$, we always have $\delta \ge 0$.
We will refer to a solution of \eqref{md} as a \emph{minimum disequilibrium} solution, whether or not it is a GNE.
\citet{harwood2024equilibrium} discuss the value of minimum disequilibrium solutions in pure Nash equilibrium problems.
See also \cite{fullerEA17} for the related concept of minimum total opportunity cost, which is the specialization of minimum disequilibrium in the context of pricing in electricity markets.

\section{Reformulation of minimum disequilibrium}
\label{sec:md_for_gne}
In this section we define an equivalent form of the minimum disequilibrium problem \eqref{md} that will be useful.
Here and in the following sections, define
\[
\hat{F} \defn \set{ (x, y_1, \dots, y_m) : (x,y_i) \in F_i, \forall i}.
\]
We then note that the feasible set of \eqref{md} is $\hat{F} \cap G$.
Define
$\pi_0 : \prod_{i=0}^m \mbb{R}^{n_i} \to \mbb{R}^{n_0}$
as the projection onto the ``$x$'' variables:
\[
\pi_0(x, y_1, \dots, y_m) = x.
\]
Then the projection of the feasible of \eqref{md} is
\begin{equation}
\label{projected_feasible_set}
	\pi_0(\hat{F} \cap G) = \set{x \in \mbb{R}^{n_0} : \exists y \text{ such that } (x,y) \in \hat{F} \cap G}.
\end{equation}
Also define $\pi_{-0} : \prod_{i=0}^m \mbb{R}^{n_i} \to \prod_{i=1}^m \mbb{R}^{n_i}$
as projection onto the ``$y$'' (everything except the ``$x$'') variables:
\[
\pi_{-0}(x, y_1, \dots, y_m) = (y_1, \dots, y_m).
\]
To avoid having to deal with extended real--valued functions, we make use of the following assumption, which essentially states that each player problem \eqref{agent_i} is bounded for all $x$ feasible for Problem~\eqref{md}.

\begin{assumption}
\label{assm:bounded}
Assume that for each $i$, $g_i^*(x) > -\infty$ for all $x \in \pi_0(\hat{F} \cap G)$.
\end{assumption}

We would like to highlight some of the structure in the problem \eqref{md} when we use the aforementioned choice of $\mu : (v_1, \dots, v_m) \mapsto \smallsum_i v_i$.
The following result establishes an equivalent reformulation of \eqref{md}.

\begin{proposition}
\label{prop:md_alt}
Let Assumption~\ref{assm:bounded} hold.
Let $\mu : \mbb{R}^m \to \mbb{R}$ be defined by $\mu(v) = \smallsum_i v_i$.
Consider
\begin{align}
\label{md_alt}
\tag{$\mathcal{MD}'$}
\inf_{x,y,w}	& \smallsum_i g_i(x,y_i)  - w \\
\st 				
\notag &(x,y) \in G, \\
\notag &(x,y) \in \hat{F}, \\
\notag & w \le \smallsum_i g_i(x,y'_i), \quad \forall y' : (x,y') \in \hat{F}.
\end{align}
The $x$ and $y$ components of the solution sets of Problems~\eqref{md} and \eqref{md_alt} are equal and the infimum of \eqref{md_alt} is $\delta$.
\end{proposition}
\begin{proof}
First, note that the constraint 
$w \le \smallsum_i g_i(x,y'_i)$, $\forall y' : (x,y') \in \hat{F}$
is equivalent to 
$w \le \inf_{y'}\set{\smallsum_i g_i(x,y_i') : (x,y') \in \hat{F}}$.
By the definition of $\hat{F}$, this minimization decomposes and so the infimum equals
$\sum_i \inf_{y_i'} \set{g_i(x,y_i') : (x,y_i') \in F_i}$.
Thus the constraint is equivalent to 
$w \le \smallsum_i g_i^*(x)$.

Now, if \eqref{md} is infeasible, it is easy to see that \eqref{md_alt} must be infeasible as well.
Conversely (via the contrapositive), if \eqref{md} is feasible, then so is \eqref{md_alt}, since since the assumption $g_i^*(x) > -\infty$ means that there must exist (finite) $w$ such that $w \le \smallsum_i g_i^*(x)$.
Consequently, if one or the other is infeasible, both solution sets are empty, their infima are both $+\infty$, and the result holds trivially.

Otherwise, assume \eqref{md_alt} is feasible and let the infimum of \eqref{md_alt} be $\tilde{\delta}$.
Choose $\epsilon > 0$ and let $(x,y,w)$ be feasible in \eqref{md_alt} with 
$\smallsum_i g_i(x,y_i) - w \le \tilde{\delta} + \epsilon$.
Since $w \le \smallsum_i g_i^*(x)$, we have
$\mu(\allg(x,y) - \allg^*(x)) = \smallsum_i g_i(x,y_i) - \smallsum_i g_i^*(x) \le \tilde{\delta} + \epsilon$.
Clearly $(x,y)$ is feasible for \eqref{md}, and so this implies $\delta \le \tilde{\delta} + \epsilon$.
This argument holds for all $\epsilon$ so $\delta \le \tilde{\delta}$.

Conversely, assume  \eqref{md} is feasible, choose $\epsilon > 0$, and let $(x,y)$ be feasible in  \eqref{md} with 
$\mu(\allg(x,y) - \allg^*(x)) = \smallsum_i (g_i(x,y_i) - g_i^*(x)) \le \delta + \epsilon$.
We can take $w = \smallsum_i g_i^*(x)$, which gives $(x,y,w)$ feasible in \eqref{md_alt} with 
$\smallsum_i g_i(x,y_i) - w \le \delta + \epsilon$.
Again, letting the infimum of \eqref{md_alt} be $\tilde{\delta}$, we have 
$\tilde{\delta} \le \delta + \epsilon$.
Again, this argument holds for all $\epsilon$ so $\tilde{\delta} \le \delta$. 
Combined with the previous inequality, we have $\delta = \tilde{\delta}$.

When a solution exists for either \eqref{md} or \eqref{md_alt}, we can essentially repeat the arguments above with $\epsilon = 0$ and use $\delta = \tilde{\delta}$ to see that a solution for one is a solution for the other.
\end{proof}

It is possible that for certain values of $x$, $\set{y' : (x, y') \in \hat{F}}$ has infinite, and in particular uncountably infinite, cardinality.
Consequently, Problem \eqref{md_alt} may be categorized as a generalized semi-infinite program (GSIP);
see \cite{guerravazquezEA08} for a review.
Although progress has been made on solving GSIP, this class of problems is intrinsically less--well behaved than ``standard'' semi-infinite programs (SIP).
In the next section we will see how the problem of finding normalized Nash equilibrium yields an SIP.

\section{Normalized Nash equilibrium}
\label{sec:nne}
%
We begin to tackle the concept of a normalized Nash equilibrium.
In this section, we adapt this definition for our framework, and try to relax the typical assumptions required.
In the following subsections, we recall the established definition of a normalized Nash equilibrium and gain some familiarity with it before specializing the minimum disequilibrium concept for normalized Nash equilibrium.

\subsection{The jointly constrained case}
For the moment, consider players in the form
\[
	\inf_{y_i} \set{ f_i(y_{-i}, y_i) : (y_{-i}, y_i) \in Y }
\]
Note that the feasible set for this player is parameterized by the other players' decisions $y_{-i}$;
however, the ``host set'' $Y$ is the same for each player.
We will refer to this situation as the ``jointly constrained case''
(previous treatments typically assume $Y$ is convex, and so often call it the jointly convex case \cite{facchinei2010generalized}).

If we use the choice
$\mu : (v_1, \dots, v_m) \mapsto \smallsum_i v_i$,
the minimum disequilibrium problem \eqref{md} for this game takes the form
\[\begin{aligned}
\delta =
\inf_{y}\; & \smallsum_i f_i(y_{-i}, y_i) - f^*(y) \\
\st
& y \in Y,
\end{aligned}\]
where
\begin{equation} \label{eq:f_star}
f^*(y) \defn \sum_i \inf_{z_i} \set{ f_i(y_{-i}, z_i) : (y_{-i},z_i) \in Y }.    
\end{equation}
Once again, this problem is equivalent to a GSIP.

The term \emph{normalized} Nash equilibrium (NNE) was originally inspired by the fact that it is a particular solution where the KKT multipliers satisfy a certain normalization condition;
in this way it is a more restrictive solution, but often one that is easier to find numerically.
A definition more relevant to the present setting is that a NNE is a point $y^*$ such that
\[
y^* \in \arg \min_y \set{\sum_i f_i(y_{-i}^*, y_i) : y \in Y}
\]
(see \cite{vonheusingerEA09,flam1994noncooperative}).
Equivalently,
$y^* \in Y$ and 
$\sum_i f_i(y_{-i}^*, y_i^*) \le \sum_i f_i(y_{-i}^*, y_i)$, for all $y \in Y$.
In the following subsection, we show that for the jointly constrained case, all NNE are GNE;
in general, the converse does not hold.

Taking it as a postulate for now, an NNE is a solution of 
\[\begin{aligned}
\delta^{\dagger} \defn
\inf_{y}\; & \smallsum_i f_i(y_{-i}, y_i) - f^{\dagger}(y) \\
\st
& y \in Y,
\end{aligned}\]
with $\delta^{\dagger} = 0$, where 
\begin{equation} \label{eq:f_dagger}
f^{\dagger}(y) \defn \inf_z \set{\smallsum_i f_i(y_{-i}, z_i) : z \in Y}.
\end{equation}
Comparing this with the minimum disequilibrium problem, the only difference is in the optimal value function that appears in the objective.
Noting that the feasible sets in the optimization problems defining $f^*$ are independent, we have
\begin{equation} \label{eq:f_star_alt}
f^*(y) = \inf_{z} \set{ \smallsum_i f_i(y_{-i}, z_i) : (y_{-i},z_i) \in Y, \forall i }.
\end{equation}
Given the similarity between the optimization problems in \eqref{eq:f_dagger} and \eqref{eq:f_star_alt}, we might suspect that, for instance, 
$f^{\dagger}(y) \le f^*(y)$ for all $y$;
in this case we would have $\delta^{\dagger} \ge \delta \ge 0$,
which would imply that an NNE, with $\delta^{\dagger} = 0$, would also be a GNE.

While this intuition is correct, showing it rigorously is complicated by the fact that there is no immediate relationship between $f^*$ and $f^{\dagger}$:
that is, the optimization problems in \eqref{eq:f_dagger} and \eqref{eq:f_star_alt} are not a restriction or relaxation of the other.
See \Cref{fig:joint};
the feasible set of \eqref{eq:f_star_alt} for a given value of $y$ is
$\prod_i \set{z_i: (y_{-i},z_i) \in Y}$,
which is not a subset of $Y$ even in simple convex cases.
However, given a value of $y$, each feasible set
$\set{z_i: (y_{-i},z_i) \in Y}$
(in \eqref{eq:f_star}, the original definition of $f^*$),
is in fact a ``subset'' of $Y$, or more accurately, a subset of the projection of $Y$ onto the appropriate subspace.
This indicates the structure of which we should try to take advantage.

\begin{figure}
\centering
\begin{tikzpicture}
\draw[red] (0,0) circle [radius=1];
\node[above,color=red] at (0,1) {$Y$};
\draw[very thick] (-1,0) -- (1,0);
\draw[very thick] (1.414/2,1.414/2) -- (1.414/2,-1.414/2);
\draw (-1,-1.414/2) rectangle (1,1.414/2) 
	node[right]{$Y_1(y_2) \times Y_2(y_1)$};
\draw[->] (-1.2,-1.2) -- (-1.2,1);
\draw[->] (-1.2,-1.2) -- (1.2,-1.2);
\draw[dashed] (1.414/2,-1.2) node[below]{$y_1$} -- (1.414/2, -1.414/2);
\draw[dashed] (-1.2,0) node[left]{$y_2$} -- (-1,0);
\end{tikzpicture}
\caption{``Jointly constrained case''.
Here, 
$Y_1(y_2) = \set{z_1 : (z_1,y_2) \in Y}$ and similarly
$Y_2(y_1) = \set{z_2 : (y_1,z_2) \in Y}$.
$Y$ is the lower-level feasible set in the definition of NNE;
$Y_1(y_2) \times Y_2(y_1)$ is the lower-level feasible set in the definition of GNE.
Note that there is not necessarily a subset/superset relationship between the lower-level problem feasible sets.}
\label{fig:joint}
\end{figure}
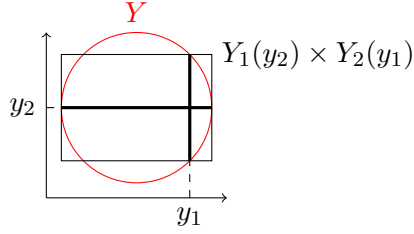

\subsection{Minimum normalized disequilibrium}
Here, we will translate the preceding discussion into our framework.
We make the following definition.

\begin{definition}
\label{defn:nne}
A point $(x^*,y^*)$ is a normalized Nash equilibrium (NNE) of the collection of \eqref{agent_i} with respect to $G$ if 
\[\begin{aligned}
	&(x^*,y^*) \in \hat{F} \cap G \quad \text{and}\\
	&y^* \in \arg\min_{y'}\set{ \smallsum_i g_i(x^*, y'_i) : y' \in \pi_{-0}(\hat{F} \cap G) }.
\end{aligned}\]
Equivalently, $(x^*,y^*) \in \hat{F} \cap G$ and
\[
	\smallsum_i g_i(x^*, y_i^*) \le \smallsum_i g_i(x^*,y'_i), \quad \forall (x',y') \in \hat{F} \cap G.
\]
\end{definition}

To try to convince ourselves that this is consistent with the definition in the previous subsection, consider a two player game in the jointly constrained case.
We define $G$ by 
$(x_1, x_2, y_1, y_2) \in G \iff x_1 = y_1, x_2 = y_2$.
We define $F_1$ and $F_2$ by
$(x,y_1) \in F_1 \iff (y_1,x_2) \in Y$ and 
$(x,y_2) \in F_2 \iff (x_1,y_2) \in Y$.
Then,
$(x,y) \in \hat{F} \cap G$ $\iff$ $(y_1,y_2) \in Y$, $(x_1,x_2) = (y_1,y_2)$.
The $x$ variables are extraneous, and the conditions in Definition~\ref{defn:nne} simplify to
$y^* \in \arg\min_y \set{ \smallsum_i g_i(y^*,y_i) : y \in Y }$.
Defining $f_i$ appropriately gives us the definition in the previous subsection.

Before relating NNE and GNE, we would like to characterize an NNE as a solution of a mathematical program satisfying certain conditions.
First, define an optimal value function relevant to the definition of NNE:
\begin{equation}
\label{eq:optimal_value_normalized}
	g^N(x) \defn \inf_{y'} \set{ \smallsum_i g_i(x, y'_i) : y' \in \pi_{-0}(\hat{F} \cap G) }.
\end{equation}
Note that an equivalent definition of $g^N$ is one with ``dummy'' variables $x'$ added to implicitly represent the projection of $\hat{F} \cap G$:
\[
	g^N(x) = \inf_{x',y'} \set{ \smallsum_i g_i(x, y'_i) : (x',y') \in \hat{F} \cap G }.
\]
Again, to avoid having to deal with extended real--valued functions, we make an assumption for $g^N$ analogous to Assumption~\ref{assm:bounded}.

\begin{assumption}
\label{assm:bounded_normal}
Assume that for all $x \in \pi_0(\hat{F} \cap G)$, $g^N(x) > -\infty$.
\end{assumption}

In the following result, we establish that an NNE is a solution of an SIP satisfying certain conditions.
Again, contrast this with the reformulation of the minimum disequilibrium problem, \eqref{md_alt}, which is a GSIP.
Since SIP tend to be better-behaved and easier to solve, NNE achieves its goal of providing a computationally easier to find solution.

\begin{proposition}
\label{prop:md_normal}
Let Assumption~\ref{assm:bounded_normal} hold.
Consider
\begin{equation} \label{md_normal}
\tag{$\mathcal{MND}$}
\begin{aligned}
\delta^{N} \defn 
\inf_{x,y,w} & \sum_i g_i(x,y_i) - w \\
\st
& (x,y) \in G, \\
& (x,y) \in \hat{F}, \\
& w \le \smallsum_i g_i(x, y'_i), \quad \forall y' \in \pi_{-0}(\hat{F} \cap G).
\end{aligned}
\end{equation}
Then $(x^*,y^*,w^*)$ is a solution of \eqref{md_normal} with $\delta^N = 0$ if and only if $(x^*,y^*)$ is an NNE.
\end{proposition}
\begin{proof}
First, note that $\delta^N \ge 0$:
Choose any feasible point $(x,y,w)$;
then we must have $(x,y) \in \hat{F} \cap G$ and $w \le g^N(x)$.
Since $(x,y) \in \hat{F} \cap G$, then $y \in \pi_{-0}(\hat{F} \cap G)$ and so by definition of $g^N(x)$, we have
$w \le g^N(x) \le \sum_i g_i(x,y_i)$.
Thus, the objective value for any feasible point $(x,y,w)$ is nonnegative.
It follows that $\delta^N$ is nonnegative.

Let $(x^*,y^*)$ be an NNE.
Then $(x^*,y^*) \in \hat{F}\cap G$ and $\smallsum_i g_i(x^*, y_i^*) \le \smallsum_i g_i(x^*,y'_i)$ for all $y' \in \pi_{-0}(\hat{F} \cap G)$.
If we let $w^* = \smallsum_i g_i(x^*, y_i^*)$, then $w^* \le \smallsum_i g_i(x^*,y'_i)$ for all $y' \in \pi_{-0}(\hat{F} \cap G)$ and so $(x^*, y^*, w^*)$ is feasible in \eqref{md_normal} with objective value zero.
It follows that $\delta^N \le 0$, but since $\delta^N \ge 0$ established previously, we have $\delta^N = 0$ and $(x^*, y^*, w^*)$ is optimal for \eqref{md_normal}.

Conversely, let $(x^*, y^*, w^*)$ be a solution of \eqref{md_normal} with $\delta^N = 0$.
Then $(x^*, y^*) \in \hat{F} \cap G$ and $w^* \le \smallsum_i g_i(x^*, y'_i)$ for all $y' \in \pi_{-0}(\hat{F} \cap G)$.
But since $\smallsum_i g_i(x^*,y_i^*) - w^* = \delta^N = 0$, we have 
$\smallsum_i g_i(x^*,y_i^*) \le \smallsum_i g_i(x^*, y'_i)$ for all $y' \in \pi_{-0}(\hat{F} \cap G)$, and it follows that $(x^*, y^*)$ is an NNE.
\end{proof}


We refer to a solution of \eqref{md_normal} as a \emph{minimum normalized disequilibrium} solution.
What is the relationship between minimum disequilibrium and minimum normalized disequilibrium?
Comparing \eqref{md_alt} and \eqref{md_normal}, the difference is in the ``infinite'' constraint;
for \eqref{md_alt} it is equivalent to $w \le \sum_i g_i^*(x)$ while
for \eqref{md_normal} it is equivalent to $w \le g^N(x)$. 
Thus, we might try to compare $\sum_i g_i^*(x)$ and $g^N(x)$.
However, in analogy with the observation in \Cref{fig:joint}, to get the definition of $g^N$, it is as if we have both relaxed the problem by allowing optimization over the dummy variables $x'$, but additionally restricted it to the set $G$.
Consequently, we need extra conditions on the sets $F_i$ and $G$ in order to relate the two problems.
In the following, define, for any index $j$,
$\pi_j : \prod_{i=0}^m \mbb{R}^{n_i} \to \mbb{R}^{n_j}$
as the projection onto the ``$y_j$'' variables:
\[
\pi_j(x, y_1, \dots, y_m) = y_j.
\]
The conditions in the next assumption are inspired by the observations we made in the jointly constrained case.

\begin{assumption}
\label{assm:supset_projection}
Assume that for any $(x,y) \in \hat{F} \cap G$, and any index $j$, we have
\[
\set{ y_j' : (x,y_j') \in F_j } \subset
\pi_j \group{ \set{ (x', y_1', \dots, y_m') \in \hat{F} \cap G : y_i' = y_i, \forall i \neq j } }.
\]
\end{assumption}

Under this assumption, we can establish that \eqref{md_normal} is a restriction of \eqref{md}.
\begin{proposition}
\label{prop:normal_restriction}
Let Assumptions~\ref{assm:bounded_normal} and \ref{assm:supset_projection} hold.
Assume that $\mu(v) = \max\set{v_1, \dots, v_m}$ in the minimum disequilibrium problem \eqref{md}.
Then, for any point $(x,y,w)$ feasible in \eqref{md_normal}, $(x,y)$ is feasible in \eqref{md} with a smaller objective value:
$\mu(g(x,y) - g^*(x)) \le \sum_i g_i(x,y_i) - w$.
Further, $\delta \le \delta^N$.
\end{proposition}
\begin{proof}
Choose $(x,y,w)$ feasible in \eqref{md_normal}.
Let the objective value be
$\epsilon \defn \sum_i g_i(x,y_i) - w$.
As noted in the proof of Proposition~\ref{prop:md_normal}, this objective value $\epsilon$ must be nonnegative.
By feasibility, we have $(x,y) \in \hat{F} \cap G$ and
\[
\smallsum_i g_i(x, y_i) \le \smallsum_i g_i(x, y_i') + \epsilon,
	\quad \forall (x',y') \in \hat{F} \cap G.
\]
Pick any $j$.
Then, this inequality also holds for all $(x',y')$ in the subset
\[
H_j \defn \set{ (x', y_1',\dots, y_m') \in \hat{F} \cap G : y_i' = y_i, \forall i \neq j }.
\]
For $(x', y') \in H_j$, $y_i' = y_i$ for each $i \neq j$, and so $g_i(x, y_i') = g_i(x, y_i)$.
Otherwise put,
\[
\smallsum_i g_i(x, y_i) \le g_j(x, y_j') + \smallsum_{i \neq j} g_i(x, y_i) + \epsilon, 
	\quad \forall (x',y') \in H_j.
\]
We can cancel terms from either side of the inequality which gives
\[
g_j(x, y_j) \le g_j(x, y_j') + \epsilon,
	\quad \forall (x',y') \in H_j.
\]
Equivalently, this inequality holds for all $y_j' \in \pi_j(H_j)$.
However, by assumption $\pi_j(H_j) \supset \set{ y_j' : (x,y_j') \in F_j }$.
Thus, the inequality still holds on this subset:
\[
g_j(x, y_j) \le g_j(x, y_j') + \epsilon,
	\quad \forall y_j' : (x,y_j') \in F_j,
\]
which implies
$g_j(x, y_j) \le g_j^*(x) + \epsilon$.
We can repeat this argument for each $j$;
thus the point $(x,y)$ is feasible in \eqref{md} with objective value 
$\max_j\set{g_j(x, y_j) - g_j^*(x)} \le \epsilon$.
This shows that \eqref{md_normal} is a restriction of \eqref{md};
it follows that $\delta \le \delta^N$.
\end{proof}

With these results, we can now relate NNE and GNE.

\begin{theorem}
\label{thm:nne_is_gne}
Let Assumptions~\ref{assm:bounded_normal} and \ref{assm:supset_projection} hold.
Then any NNE is a GNE.
\end{theorem}
\begin{proof}
Let $(x^*,y^*)$ be an NNE.
Then by Proposition~\ref{prop:md_normal}, there exists $w^*$ such that $(x^*,y^*,w^*)$ is a solution of \eqref{md_normal} with $\delta^N = 0$.
By Proposition~\ref{prop:normal_restriction}, this implies $(x^*,y^*)$ is feasible in \eqref{md} with $\delta = 0$.
Finally, since the choice of $\mu(v) = \max\set{v_1, \dots, v_m}$ satisfies the conditions required of $\mu$ in Proposition~\ref{prop:equilibrium}, this means $(x^*,y^*)$ is a GNE.
\end{proof}
%

Assumption~\ref{assm:bounded_normal} does not seem to be strictly essential to show that NNE are GNE;
however, it is a mild enough assumption and we find Theorem~\ref{thm:nne_is_gne} satisfactory.
Meanwhile, Assumption~\ref{assm:supset_projection} does not explicitly impose any specific structure on $F_i$ or $G$, although practically we might only expect the condition to hold when these sets take the form implied when converting a game from the jointly constrained case.
In this situation, $x = (x_1, \dots, x_m)$, and these variables are merely a copy of the $y$ variables, and $G$ is used to enforce this.
We establish this special case in the following result.

\begin{proposition}
\label{prop:jcc_conversion}
Assume that $n_0 = \sum_{i=1}^m n_i$ and
\[
	G = \set{ (x_1, \dots, x_m, y_1, \dots, y_m) : x_i = y_i, \forall i }.
\]
Further, assume that there exists 
$Y \subset \prod_{i} \mbb{R}^{n_i}$ such that for each $i$,
\[
	F_i = \set{ (x_1, \dots, x_m, y_i) : (x_1, \dots, x_{i-1}, y_i, x_{i+1}, \dots x_m) \in Y}.
\]
Then, Assumption~\ref{assm:supset_projection} holds.
\end{proposition}
\begin{proof}
Choose any $(x,y) \in \hat{F} \cap G$ and any index $j$.
Define
\[
	H_j \defn \set{ (x', y_1', \dots, y_m') \in \hat{F} \cap G : y_i' = y_i, \forall i \neq j }.
\]
Choose $y_j'$ such that $(x,y_j') \in F_j$.
We need to show that there exist $x'$ and $y_i'$, $i \neq j$, such that $(x',y') \in H_j$.
To this end, let $y_i' = y_i$, for each $i \neq j$, and $x' = y'$.
We establish that this point satisfies $(x',y') \in H_j$.

Clearly the conditions $y_i' = y_i$, for each $i \neq j$ are satisfied, and since $x' = y'$, we have $(x',y') \in G$.
It remains to show that $(x',y') \in \hat{F}$.
To show this, note that $y' \in Y$ is equivalent to $(y',y') \in \hat{F}$ 
($y' \in Y \iff (y',y_i') \in F_i, \forall i \iff (y',y') \in \hat{F}$).
Since we chose $y_j'$ such that $(x,y_j') \in F_j$, then forgiving a re-ordering of variables, this means $(x_{-j}, y_j') \in Y$.
But we also chose $(x,y) \in G$, so $x_{-j} = y_{-j} = y_{-j}'$, and so we have $y' \in Y$.
It follows that $(x',y') = (y',y') \in \hat{F}$.
Thus, we have $(x',y') \in H_j$, which implies $y_j' \in \pi_j(H_j)$.
Since $y_j'$ was arbitrary this establishes $\set{ y_j' : (x,y_j') \in F_j } \subset \pi_j(H_j)$.
\end{proof}

Proposition~\ref{prop:jcc_conversion} and Theorem~\ref{thm:nne_is_gne} establish that NNE are GNE in the jointly constrained case, as promised.
More generally, the minimum normalized disequilibrium problem promises to be computationally easier than the minimum disequilibrium problem, provides an approximation to a minimum disequilibrium solution, and even provides an upper bound on minimum disequilibrium
(albeit with slightly different measures of disequilibrium).



\subsection{Solution method for normalized Nash equilibrium}
\label{sec:solution_methods}
\Cref{alg:CG} solves Problem~\eqref{md_normal} to $\varepsilon$-optimality under mild conditions
(e.g.~compactness of $G$, $F_i$, and continuity of $g_i$).
This algorithm relies on the iterative solution of the optimization problem \eqref{eq:optimal_value_normalized} defining $g^N(x)$,
\[
	g^N(x) = \inf_{y'} \set{ \smallsum_i g_i(x,y'_i) : y' \in \pi_{-0}(\hat{F} \cap G) },
\]
and
\begin{align} \label{mnd_lb}
\delta^{L} \defn 
\inf_{x,y,w} & \sum_i g_i(x,y_i) - w \\
\st
\notag & (x,y) \in G, \\
\notag & (x,y) \in \hat{F}, \\
\notag & w \le \smallsum_i g_i(x, y'_i), \quad \forall y' \in F^L
\end{align}
for some finite set $F^L$ that is a subset of $\pi_{-0}(\hat{F} \cap G)$.
A key observation is that since $F^L \subset \pi_{-0}(\hat{F} \cap G)$,
$\inf \set{\smallsum_i g_i(x, y'_i) : y' \in F^L} \ge \inf \set{\smallsum_i g_i(x, y'_i) : y' \in \pi_{-0}(\hat{F} \cap G)}$.
Thus, \eqref{mnd_lb} is a relaxation of \eqref{md_normal}, and $\delta^L$ is a lower bound on $\delta^N$.
Meanwhile an upper bound on $\delta^N$ is given by
$\sum_i g_i(x,y_i) - g^N(x)$
for some $(x,y) \in \hat{F} \cap G$.
Again, these bounds converge under mild assumptions (to be more specific, compactness of $\hat{F} \cap G$ and continuity of each $g_i$);
for details see the proofs in \cite{harwood2022equilibrium,harwood2024equilibrium}, which also allow for inexact solution of the subproblems.
See also the general algorithm for solving SIP from \cite{blankenshipEA76}.

\begin{algorithm}
\caption{Solution method for Problem~\eqref{md_normal}}
\label{alg:CG}
\begin{algorithmic}[1]
\REQUIRE
$\varepsilon>0$,
$F^L \subset \hat{F} \cap G$, 
$F^L \neq \emptyset$
\STATE{ $\delta^U = +\infty$ }
\LOOP
	\STATE
    \label{step:lb}
	Solve Problem~\eqref{mnd_lb} to obtain 
		optimal solution $(x,y,w)$ and
		value $\delta^L$.
	\STATE
    \label{step:llp}
	For this value of $x$, 
		solve lower-level problem \eqref{eq:optimal_value_normalized} to obtain 
		optimal solution $y'$ and value $g^N(x)$.
	
	\IF{$w \le g^N(x)$}
    \label{step:equilibrium_found}
		\STATE $(x^*,y^*) \gets (x,y)$
		\RETURN $(x^*,y^*)$
	\ELSE
		\STATE
		$F^L \gets F^L \cup \{y'\}$
	\ENDIF
	\IF{$\sum_i g_i(x,y_i) - g^N(x) < \delta^U$}
		\STATE
		\label{step:upperbound}
			$\delta^U \gets \sum_i g_i(x,y_i) - g^N(x)$
		\STATE
			$(x^*,y^*) \gets (x,y)$
	\ENDIF
	\IF{$\delta^U - \max\set{\delta^L,0} < \varepsilon$}
	\label{step:term}
		\RETURN $(x^*,y^*)$
	\ENDIF
\ENDLOOP
\end{algorithmic}
\end{algorithm}

\section{Numerical examples}
\label{sec:examples}
In this section we present two simple examples to better demonstrate the numerical method.

\subsection{A two-player game with nonconvex players}
\label{sec:another_example}

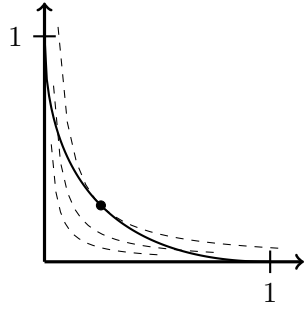
\begin{figure}[t]
    \centering
    \begin{tikzpicture}[xscale=3, yscale=3]
    \draw[->,draw=black,very thick] (0,0) -- (0,1.15);
    \draw[->,draw=black,very thick] (0,0) -- (1.15,0);
    \draw[domain=0:1,samples=100,thick] plot (\x, {(1 - sqrt(\x))^2});
    \draw[thick] (1,0.05) -- (1,-0.05) node[below] {$1$};
    \draw[thick] (0.05,1) -- (-0.05,1) node[left] {$1$};
    \draw[domain=0.06:1.04,dashed] plot (\x, {(1/16)/\x});
    \draw[domain=0.04:0.78,dashed] plot (\x, {(1/32)/\x});
    \draw[domain=0.03:0.52,dashed] plot (\x, {(1/64)/\x});
    \draw[fill=black] (0.25, 0.25) circle(0.02);
    \end{tikzpicture}
    \caption{
        Sketch of joint problem~\eqref{model:example1_joint_problem} for example of \Cref{sec:another_example}.
        Contours of the objective function are dashed lines.
    }
    \label{fig:another_example}
\end{figure}

We consider a two-player game defined by nonconvex players controlling a single continuous decision variable
\begin{equation*}
\begin{aligned}
g_1^*(x) &= \min_{y_1} \set{ -2 y_1 x_2 : y_1^r + x_2^r \leq 1, y_1 \in [0,1] }, \\
g_2^*(x) &= \min_{y_2} \set{  x_1 y_2 : x_1^r + y_2^r \leq 1, y_2 \in [0,1] }, 
\end{aligned}
\end{equation*}
where
$r = \tfrac{1}{2}$,
and
$G = \set{(x_1, x_2, y_1, y_2) : x_1 = y_1, x_2 = y_2}$.
Note that each player minimizes a linear function subject to a nonconvex constraint since $r \in (0,1)$.
The above formulation is simply putting a generalized game between the two players
\begin{equation*}
\begin{aligned}
& \min_{y_1} \set{ -2y_1 y_2 : y_1^r + y_2^r \leq 1, y_1 \in [0,1] }, \\
& \min_{y_2} \set{  y_1 y_2 : y_1^r + y_2^r \leq 1, y_2 \in [0,1] }, 
\end{aligned}
\end{equation*}
into the form required by \Cref{defn:gne}.

To initialize the $F^L$ set in \Cref{alg:CG}, a reasonable approach might be to solve the ``joint'' optimization problem
(combining the constraints and summing the objectives of the players in their simplified form):
\begin{equation} \label{model:example1_joint_problem}
\min_{y} \set{ - y_1 y_2 : y_1^r + y_2^r \leq 1, y \in [0,1]^2 }.
\end{equation}
The optimal solution of this problem is $\group{\tfrac{1}{4}, \tfrac{1}{4}}$
(see \Cref{fig:another_example}),
which we use to initialize the set $F^L$.

The lower bounding problem, with the single cut from $F^L$, is
\begin{equation*}
\begin{aligned}
\min_{x,y,w}\; & -2y_1 x_2 + x_1 y_2 - w \\
\st
& w \le -2y_1' x_2 + x_1 y_2', \quad \forall y' \in \set{(\sfrac{1}{4}, \sfrac{1}{4})}, \\
& y_1^r + x_2^r \le 1, \\
& x_1^r + y_2^r \le 1, \\
& y \in [0,1]^2, \\
& x_1 = y_1, x_2 = y_2.
\end{aligned}
\end{equation*}
However, this simplifies to
\[
\min_{y_1, y_2, w} \set{
    -y_1 y_2 - w : 
    w \le -(\sfrac{1}{2}) y_2 + (\sfrac{1}{4}) y_1,
    y_1^r + y_2^r \le 1,
    y \in [0,1]^2
}.
\]
The solution is $(1, 0, \sfrac{1}{4})$ with optimal value $-\frac{1}{4}$.
The lower-level problem, using the dummy variable formulation to implicitly represent the projection, is
\begin{equation*}
\begin{aligned}
g^N(1,0) =
\min_{x_1',x_2', y_1',y_2'}\; & -2 (y_1' \cdot 0) + 1 \cdot y_2' \\
\st
& (y_1')^r + (x_2')^r \le 1, \\
& (x_1')^r + (y_2')^r \le 1, \\
& y' \in [0,1]^2, \\
& x_1' = y_1', x_2' = y_2',
\end{aligned}
\end{equation*}
which also simplifies to
\[
g^N(1,0) = \min_{y_1',y_2'} \set{ y_2' : (y_1')^r + (y_2')^r \le 1, y' \in [0,1]^2 }.
\]
The set of optimal solutions to $g^N$ are $\{(y_1',y_2') = (\gamma, 0) : \gamma \in [0,1]\}$ with optimal value $g^N(\gamma,0) = 0$.
We have $\frac{1}{4} = w \not \le g^N(\gamma,0) = 0$, implying that we should update $F^L$. However, the upper bound becomes $\delta^U = -2(1\cdot 0) + (1\cdot 0) - g^N(\gamma,0) = 0$, indicating that the disequilibrium is zero and that an equilibrium has been found.
Thus, we terminate at Step~\ref{step:term} in \Cref{alg:CG} with solution $y^* = (1,0)$ (and $x^* = y^*$).
\Cref{alg:CG} only requires one full iteration to terminate.

\subsection{Game with discretely-constrained Nash-Cournot knapsack players}
\label{sec:knapsack}

Consider a multi-player discretely-constrained Nash-Cournot game in which player $j \in \mc{J}$ solves
\begin{subequations} \label{model:knapsack_NCgame}
\begin{align}
\min_{y_j}\; &
    \sum_{l \in \mc{L}_j}
        \group{ 
            c_{jl}
            - \group{\alpha_l - \beta_l \smallsum_{i \in \mc{J}} y_{il}}
        } 
        y_{jl} \\
\st
&\sum_{l \in \mc{L}_j} a_{jl} y_{jl} \leq b_j, \\
\label{con:shared}
&\sum_{i \in \mc{J}} d_{il} y_{il} \leq e_l, \quad \forall l \in \mc{L}, \\
&y_j \in \set{0,1}^{n_j},
\end{align}
\end{subequations}
where
$\alpha_l > 0$ and $\beta_l > 0$ represent the parameters of a linear inverse demand function of market $l$,
$\mc{L}_j$ denotes the set of markets available to player $j$, 
$c_{jl}$ and $a_{jl}$ denote the cost incurred and resources consumed, respectively, when player $j$ chooses to participate in market $l$, and 
$b_j$ is the total resource amount available to player $j$.
Meanwhile, constraints~\eqref{con:shared} are a set of ``shared'' constraints, known to all players, which limit the total resources (summed over all players) that can be consumed in market $l$.
\citet{papageorgiou2023pooling} described a variant of this game without shared constraints.
Each player has binary decision variables and solves a multi-knapsack problem. Note that, while each player's objective function has a quadratic term $y_{jl}^2$, it reduces to a linear objective since $y_{jl}^2=y_{jl}$ when $y_{jl} \in \{0,1\}$. 

First, for simplicity, assume that $\mc{L}_j = \mc{L}$ for all $j$;
that is, each player participates in each market.
Second, rather than put this game into the form required by \Cref{defn:gne}, we follow the logic from the example of \Cref{sec:another_example} to see that the lower bounding problem is
\begin{equation} \label{knapsack:lbp}
\begin{aligned}
\min_{y, w}\; & 
\sum_{j \in \mc{J}}
    \sum_{l \in \mc{L}}
        \group{ 
            c_{jl}
            - \alpha_l
            + \beta_l \smallsum_{i \in \mc{J}} y_{il} 
        } 
        y_{jl}
- w \\
\st
&\smallsum_{l \in \mc{L}} a_{jl} y_{jl} \leq b_j, \quad \forall j \in \mc{J}, \\
&\smallsum_{j \in \mc{J}} d_{jl} y_{jl} \leq e_l, \quad \forall l \in \mc{L}, \\
&y_{jl} \in \set{0,1}, \quad \forall (j,l) \in \mc{J} \times \mc{L}, \\
&w \le 
\smallsum_{j \in \mc{J}} \smallsum_{l \in \mc{L}}
    \group{
        \group{
            c_{jl}
            - \alpha_l
            + \beta_l \smallsum_{i \neq j} y_{il} 
        } 
        y_{jl}'
        + \beta_l (y_{jl}')^2
    },
    \quad \forall y' \in F^L.
\end{aligned}
\end{equation}
It is easy to see that the lower-bounding problem~\eqref{knapsack:lbp} is linearly constrained;
introducing the variables $z_l = \sum_{i} y_{il}$ and writing the objective as
\[
\smallsum_{j,l} c_{jl} y_{jl}
+
\smallsum_{l}
    \group{
        - \alpha_l
        + \beta_l z_l
    } 
    z_l
- w
\]
it is clear that the lower bounding problem may be solved as a mixed-integer (convex) quadratic program (MIQP).
Meanwhile the lower-level problem is
\begin{equation} \label{knapsack:llp}
\begin{aligned}
g^N(y) =
\min_{y'}\; & 
\sum_{j \in \mc{J}}
    \sum_{l \in \mc{L}} \group{ 
        \group{
            c_{jl}
            - \alpha_l
            + \beta_l \smallsum_{i \neq j} y_{il} 
        } 
        y_{jl}'
        + \beta_l(y_{jl}')^2
    } \\
\st
&\smallsum_{l \in \mc{L}} a_{jl} y_{jl}' \leq b_j, \quad \forall j \in \mc{J}, \\
&\smallsum_{j \in \mc{J}} d_{jl} y_{jl}' \leq e_l, \quad \forall l \in \mc{L}, \\
&y_{jl}' \in \set{0,1}, \quad \forall (j,l) \in \mc{J} \times \mc{L}.
\end{aligned}
\end{equation}
Recalling that $x^2 = x$ when $x \in \set{0,1}$, we can write the objective of the LLP~\eqref{knapsack:llp} as 
\[
\smallsum_{j \in \mc{J}}
    \smallsum_{l \in \mc{L}}
    \group{
        c_{jl}
        - \alpha_l 
        + \beta_l
        + \beta_l \smallsum_{i \neq j} y_{il}
    } 
    y_{jl}'.
\]
Consequently, it is easy to see that the LLP is a 0--1 integer linear program.

For our computational experiments, we generate 100 random instances involving $|\mc{J}|$ players and $|\mc{L}|$ markets for each choice of $|\mc{J}| \in \set{5,25,50}$ and $|\mc{L}| \in \set{10,20}$.
Thus, we solve 600 total instances.
Let $\gamma=10^3$ and assume that the function $\texttt{randint}(a,b)$ generates a random integer between $a$ and $b$, inclusive.
Then, we generate random (integer) parameters for multi-player knapsack instances as follows:
$\alpha_l = \texttt{randint}(1,|\mc{J}|\gamma), \forall l \in \mc{L}$;
$\beta_l = \texttt{randint}(1,\gamma), \forall l \in \mc{L}$;
$c_{jl} = \texttt{randint}(1,\gamma), \forall (j,l) \in \mc{J} \times \mc{L}$;
$a_{jl} = \texttt{randint}(1,\gamma), \forall (j,l) \in \mc{J} \times \mc{L}$;
$b_{j} = \max\set{a_{jl} : l \in \mc{L}}, \forall j \in \mc{J}$;
$d_{jl} = \texttt{randint}(1,\gamma), \forall (j,l) \in \mc{J} \times \mc{L}$;
$e_{l} = \max\set{ d_{jl} : j \in \mc{J} }, \forall l \in \mc{L}$.
Note that these values of $b_j$ and $e_l$ ensure that each player can participate in any market.
We formulate each instance using \texttt{cvxpy} \cite{diamond2016cvxpy} and solve them using \texttt{pyscipopt} \cite{MaherMiltenbergerPedrosoRehfeldtSchwarzSerrano2016}.
As in the previous example, we initialize the set $F^L$ by solving the ``joint'' optimization problem in which we combine/concatenate the constraints and sum the objectives of the players in their simplified form.
We set $\epsilon = 0.01$.

\Cref{fig:hist-knapsack} shows histograms of the number of iterations required to reach a normalized Nash equilibrium.
The first observation is that all instances have an equilibrium
(or, more precisely, a point with normalized disequilibrium below the tolerance $\epsilon$).
On these instances, \Cref{alg:CG} typically only needs one iteration to converge;
however, more than one iteration is required on several instances, but rarely more than ten iterations.

\begin{figure} [h!]
\centering
\begin{subfigure}[b]{0.49\textwidth}
\includegraphics[width=\textwidth]{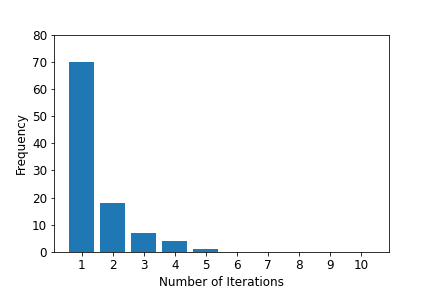}
\caption{5,10}
\end{subfigure}
\begin{subfigure}[b]{0.49\textwidth}
\includegraphics[width=\textwidth]{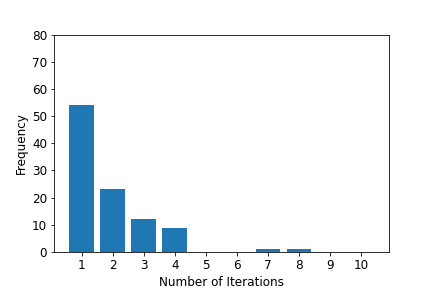}
\caption{5,20}
\end{subfigure}
\begin{subfigure}[b]{0.49\textwidth}
\includegraphics[width=\textwidth]{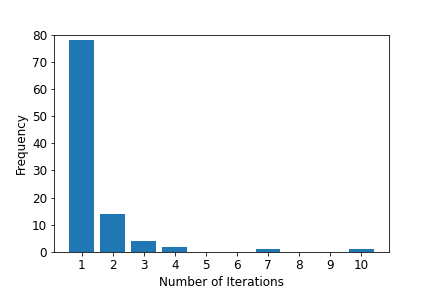}
\caption{25,10}
\end{subfigure}
\begin{subfigure}[b]{0.49\textwidth}
\includegraphics[width=\textwidth]{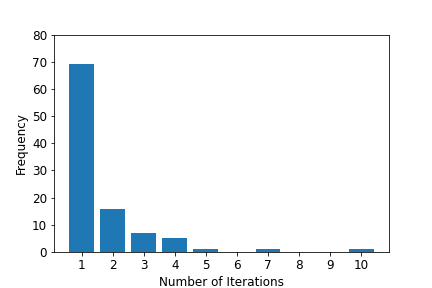}
\caption{25,20}
\end{subfigure}
\begin{subfigure}[b]{0.49\textwidth}
\includegraphics[width=\textwidth]{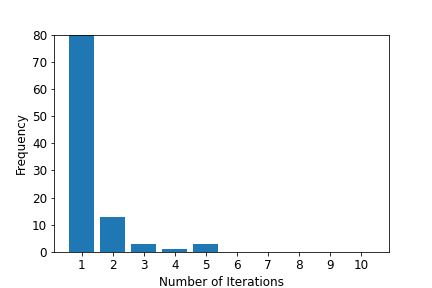}
\caption{50,10}
\end{subfigure}
\begin{subfigure}[b]{0.49\textwidth}
\includegraphics[width=\textwidth]{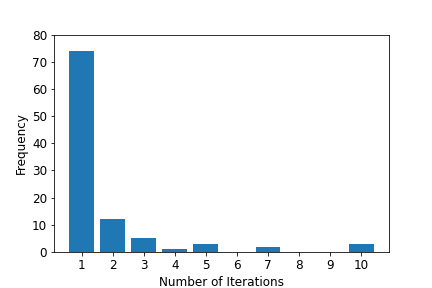}
\caption{50,20}
\end{subfigure}
\caption{Histograms of the number of major iterations needed by Algorithm~\ref{alg:CG} to converge to a normalized Nash equilibrium. Each subcaption lists the (number of players, number of markets) tuple $|\mc{J}|,|\mc{L}|$. 100 instances were solved for each $|\mc{J}|,|\mc{L}|$ pair.}
\label{fig:hist-knapsack}
\end{figure}

\Cref{fig:knapsack_solution_time} shows the mean and max time per iteration over 100 instances for different player-market configurations of $|\mc{J}|$ and $|\mc{L}|$.
To be clear, one iteration includes all, or (if a termination criterion is satisfied) a subset, of the steps in the main \textbf{loop} of \Cref{alg:CG}.
As expected, the time per iteration increases in the number $|\mc{J}|$ of players and the number $|\mc{L}|$ of markets.
Even though the lower bounding problem~\eqref{knapsack:lbp} and the lower-level problem~\eqref{knapsack:llp} each have $|\mc{J}||\mc{L}|$ binary decision variables, the solver has no difficulty optimizing these instances even when there are 1000 binary decision variables with $|\mc{J}|=50$ and $|\mc{L}|=20$.    

\begin{figure} 
\centering
\includegraphics[width=12cm]{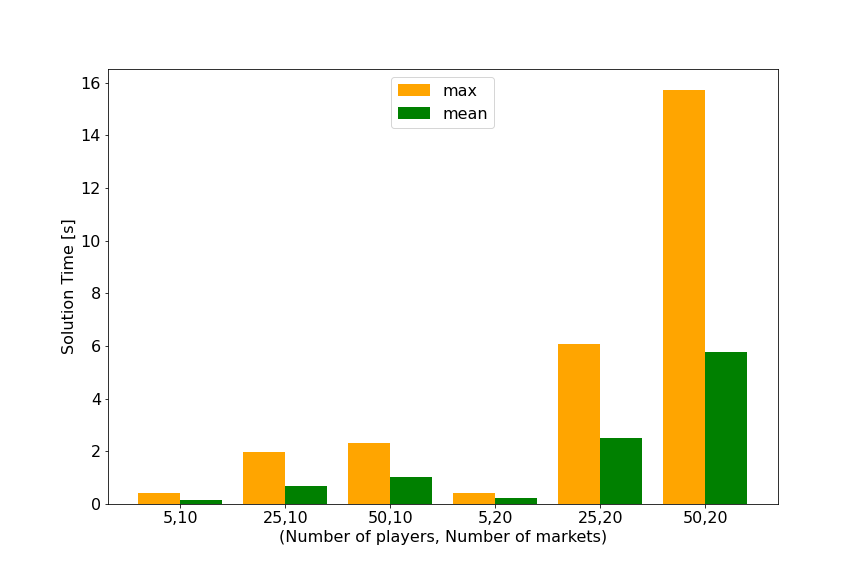}
\caption{Max and mean solution time per iteration for multi-player knapsack instances. }
\label{fig:knapsack_solution_time}
\end{figure} 

We point out that it is a somewhat surprising result that every instance has an equilibrium;
we are not aware of any theoretical results that would explain why this general class of games (of which we have generated 600 random instances) would have an equilibrium.
The work by \citet{harks2024generalized} seems the most relevant;
the essence of those results is that a point is an equilibrium if and only if it is an equilibrium for a game where the players are replaced by versions involving the convex hull of their strategy sets.
It is plausible that those results apply to this game involving knapsack players.
This merits future investigation.

As far as alternative methods are concerned, in the appendix, we show that a popular framework, known as the complementarity approach, need not produce a Nash equilibrium to this game due to the players' discrete (hence, nonconvex) decisions.
Even the numerical methods proposed in \cite{harks2024generalized} do not quite apply, as they focused on special cases, such as when a tight formulation of the players is known a priori
(that is, when the continuous relaxations of the players' strategy sets coincides with their convex hulls).

\section{Conclusions and future research directions} \label{sec:conclusions}
In this work we have extended the minimum disequilibrium concept from \cite{harwood2024equilibrium} to minimum normalized disequilibrium.
A minimum normalized disequilibrium yields a normalized Nash equilibrium when the latter exists.
These results have been established without need to assume any convex structure in the players.
Just as normalized Nash equilibrium is a potentially easier-to-find solution, minimum normalized disequilibrium is computationally more tractable to find than minimum disequilibrium.
Moreover, our approach is exact; it is not a heuristic.

There are numerous opportunities for future research.
Revisiting the theory of \cite{harks2024generalized} would be worthwhile;
whereas that theory might identify classes of nonconvex games that have an equilibrium, the algorithm proposed in this work would actually be able to find them.
Extending the proposed method to find generalized Nash equilibria would also be valuable.
It may be beneficial to create a hybrid approach in which heuristic methods are coupled with our exact approach as was done in \citep{papageorgiou2023pooling}.
Distributed algorithms for GNE and NNE problems could also apply in our setting.

\appendix
\section{Failure of complementarity for nonconvex NNE}


A popular approach to solving convex games is via complementarity, whereby one writes the KKT conditions for each player, concatenates all of these conditions, and then invokes an algorithm, e.g.~the PATH method \cite{dirkse1995path}, to solve the resulting complementarity system.
However, when the convexity assumption does not hold, this approach is an unreliable heuristic at best.
In this section, we show that the complementarity approach need not produce a generalized Nash equilibrium when players have discrete decisions.
Indeed, we show that \textit{any} feasible solution would satisfy the KKT conditions of the complementarity approach.  

For concreteness, we revisit the example from \Cref{sec:knapsack}.
To arrive at a complementarity formulation, we take advantage of two insights. 
First, we replace each binary variable $y_{jl}$ with a continuous variable $y_{jl} \in [0,1]$ subject to the additional complementarity constraint $y_{jl}(1-y_{jl}) = 0$. 
Thus, we have players defined by continuous nonlinear programs with differentiable objectives and constraint functions.
The KKT conditions are at least necessary for an optimum of such problems, but, as we shall see, trivially so.
Second, with all decision variables now being continuous and all but the complementarity constraints $y_{jl}(1-y_{jl}) = 0$ being linear, we follow a complementarian's approach by attempting to exploit a result from Harker \cite{harker1991generalized} who demonstrated that, when all constraints are linear, a ``VI solution'' (in our case, a complementarity solution) can be obtained by giving the same dual variables to the common constraints. 
Applying these two results leads to the nonconvex continuous formulation for each player $j$:
\begin{subequations} \label{model:knapsack_nlp_NCgame}
\begin{align}
p_j^*(y_{-j}) = 
\max_{y_j}\;& 
    \sum_{l \in \mc{L}_j} \group{ 
        \group{ 
            \alpha_l - \beta_l \smallsum_{i \in \mc{J}} y_{il} 
        } 
        y_{jl} - c_{jl} y_{jl} 
    }
    = p_j(y_{j},y_{-j})
    & & \quad \textrm{\underline{Dual vars}} \\ 
\st
& \sum_{l \in \mc{L}_j} a_{jl} y_{jl} \leq b_j,
    & & \quad \pi_j \ge 0, \\
& y_{jl}(1 - y_{jl}) = 0,
    & & \quad \gamma_{jl} \in \mbb{R}, \quad \forall l \in \mc{L}_j, \\	
& -y_{jl} \ge -1,
    & & \quad \mu_{jl} \ge 0, \quad \forall l \in \mc{L}_j, \\
& y_{jl} \ge 0,
    & & \quad \nu_{jl} \ge 0, \quad \forall l \in \mc{L}_j, \\
& \sum_{i \in \mc{J}} d_{il} y_{il} \le e_l,
    & & \quad \omega_{l} \ge 0, \quad \forall l \in \mc{L}.
\end{align}
\end{subequations}

The corresponding Lagrangian function for each player $j$ is
\begin{multline*}
L_j(y_j, \pi_j, \gamma_j, \mu_j, \nu_j) =
\sum_{l \in \mc{L}_j} \group{
    \group{
        \alpha_l - \beta_l \smallsum_{i \in \mc{J}} y_{il}
    } y_{jl} - c_{jl} y_{jl}
}
+ \pi_j \group{ b_j - \smallsum_{l \in \mc{L}_j} a_{jl} y_{jl} }
\\ 
+ \sum_{l \in \mc{L}_j} \gamma_j y_{jl}(1 - y_{jl})
+ \sum_{l \in \mc{L}_j} \mu_{jl} (1 - y_{jl})
+ \sum_{l \in \mc{L}_j} \nu_{jl} y_{jl}
+ \sum_{l \in \mc{L}} \omega_{l} \group{ e_l - \smallsum_{i \in \mc{J}} d_{il} y_{il} }. 
\end{multline*}

The necessary, but not sufficient, first-order KKT optimality conditions associated with the nonlinear model~\eqref{model:knapsack_nlp_NCgame} when all players are represented by a monolithic system of equations are
\begin{subequations} \label{model:knapsack_nlp_kkt_NCgame}
\begin{align}
    \alpha_l - c_{jl} - \beta_l \sum_{i \neq j} y_{il} - 2 \beta_l y_{jl} 
    - \pi_j a_j
    + \gamma_{jl}
    - 2\gamma_{jl} y_{jl}
    - \mu_{jl}
    + \nu_{jl}
    - \omega_l d_{jl} = 0,
        &\quad \forall l \in \mc{L}_j, j \in \mc{J}, \\ 
0 \le b_j - \sum_{l \in \mc{L}_j} a_{jl} y_{jl} \perp \pi_j \ge 0,
        &\quad \forall j \in \mc{J}, \\
\gamma_{jl} y_{jl}(1 - y_{jl}) = 0,
        &\quad \forall l \in \mc{L}_j, j \in \mc{J}, \\ 
0 \le 1 - y_{jl} \perp \mu_{jl} \ge 0,
        &\quad \forall l \in \mc{L}_j, j \in \mc{J}, \\ 
0 \le y_{jl} \perp \nu_{jl} \ge 0,
        &\quad \forall l \in \mc{L}_j, j \in \mc{J}, \\ 
0 \le e_l - \sum_{i \in \mc{J}} d_{il} y_{il} \perp \omega_l \ge 0,
        &\quad \forall l \in \mc{L}.
\end{align}
\end{subequations}

The following proposition shows that if one were to choose \emph{any} feasible solution $\hat{y}_j$ to player $j$'s binary knapsack problem~\eqref{model:knapsack_NCgame} and concatenate all such solutions into a single decision vector
$\hat{y}=(\hat{y}_1,\dots,\hat{y}_N) \in \set{0,1}^{\sum_j n_j}$ satisfying $\sum_{i \in \mc{J}} d_{il} \hat{y}_{il} \le e_l$ for all $l \in \mc{L}$, then $\hat{y}$ trivially satisfies the KKT conditions~\eqref{model:knapsack_nlp_kkt_NCgame}.
More importantly, $\hat{y}$ need not satisfy the requirement $p_j^*(\hat{y}_{-j}) \ge p_j(y_j, \hat{y}_{-j})$
for all $j \in \mc{J}$ and $y_j \in \set{0,1}^{n_j}$ satisfying 
$\sum_{l \in \mc{L}_j} a_{jl} \hat{y}_{jl} \leq b_j$
and
$\sum_{i \in \mc{J}} d_{il} y_{il} \le e_l$
for all $l \in \mc{L}$, which, of course, is the definition of a generalized Nash equilibrium.
In short, solving \eqref{model:knapsack_nlp_kkt_NCgame} directly gives no guarantee that an equilibrium has been found.

\begin{proposition}
Let $\hat{y}_j$ be any feasible solution to player $j$'s binary knapsack problem~\eqref{model:knapsack_NCgame} such that
$\hat{y} = (\hat{y}_1, \dots, \hat{y}_N)$
satisfies
$\sum_{i \in \mc{J}} d_{il} \hat{y}_{il} \le e_l$
for all $l \in \mc{L}$.
Then there exists 
$(\hat{\pi},\hat{\gamma},\hat{\mu},\hat{\nu},\hat{\omega})$ 
such that 
$(\hat{y},\hat{\pi},\hat{\gamma},\hat{\mu},\hat{\nu},\hat{\omega})$
is a feasible solution to the KKT conditions \eqref{model:knapsack_nlp_kkt_NCgame}. 
\end{proposition}
\begin{proof}
Consider any feasible binary solution
$\hat{y}=(\hat{y}_1, \dots, \hat{y}_N) \in \{0,1\}^{\sum_j n_j}$ 
satisfying
$\sum_{l \in \mc{L}_j} a_{jl} \hat{y}_{jl} \le b_j$
for all $j \in \mc{J}$ and
$\sum_{i \in \mc{J}} d_{il} \hat{y}_{il} \leq e_l$
for all $l \in \mc{L}$.
For all $j \in \mc{J}$, let $\hat{\pi}_j = 0$ if
$\sum_{l \in \mc{L}_j} a_{jl} \hat{y}_{jl} < b_j$
and $\hat{\pi}_j \ge 0$ otherwise.
For all $l \in \mc{L}$, let $\hat{\omega}_l = 0$ if
$\sum_{i \in \mc{J}} d_{il} \hat{y}_{il} < e_l$
and $\hat{\omega}_l \ge 0$ otherwise.
If $\hat{y}_{jl} = 0$, let $\hat{\mu}_{jl} = 0$, and let $(\hat{\nu}_{jl}, \hat{\gamma}_{jl})$
satisfy
$\hat{\nu}_{jl} \ge 0$ and
$\hat{\gamma}_{jl} = -\group{\alpha_l - c_{jl} - \beta_l \sum_{i \neq j} \hat{y}_{il} - \hat{\pi}_j a_j + \hat{\nu}_{jl} - \hat{\omega}_l d_{jl}}$.  
If $\hat{y}_{jl} = 1$, let $\hat{\nu}_{jl} = 0$, and let $(\hat{\mu}_{jl},\hat{\gamma}_{jl})$
satisfy
$\hat{\mu}_{jl} \ge 0$ and
$\hat{\gamma}_{jl} = \alpha_l - c_{jl} - \beta_l \sum_{i \neq j} \hat{y}_{il} - 2\beta_l - \hat{\pi}_j a_j - \hat{\mu}_{jl} - \hat{\omega}_l d_{jl}$.
Then $(\hat{y},\hat{\pi},\hat{\gamma},\hat{\mu},\hat{\nu},\hat{\omega})$ is a solution of the system of equations~\eqref{model:knapsack_nlp_kkt_NCgame}.
\end{proof}


\bibliography{main}

@article{abada2013generalized,
  title     = {A generalized {N}ash--{C}ournot model for the northwestern {E}uropean natural gas markets with a fuel substitution demand function: {T}he {GaMMES} model},
  author    = {Abada, Ibrahim and Gabriel, Steven and Briat, Vincent and Massol, Olivier},
  journal   = {Networks and Spatial Economics},
  volume    = {13},
  pages     = {1--42},
  year      = {2013},
  publisher = {Springer}
}

@article{arrowEA54,
  title     = {Existence of an equilibrium for a competitive economy},
  author    = {Arrow, Kenneth J. and Debreu, Gerard},
  journal   = {Econometrica: Journal of the Econometric Society},
  pages     = {265--290},
  year      = {1954},
  publisher = {JSTOR}
}

@article{ba2022exact,
  title={Exact penalization of generalized {N}ash equilibrium problems},
  author={Ba, Qin and Pang, Jong-Shi},
  journal={Operations Research},
  volume={70},
  number={3},
  pages={1448--1464},
  year={2022},
  publisher={INFORMS}
}

@article{blankenshipEA76,
  title     = {Infinitely constrained optimization problems},
  author    = {Blankenship, Jerry W. and Falk, James E.},
  journal   = {Journal of Optimization Theory and Applications},
  volume    = {19},
  number    = {2},
  pages     = {261--281},
  year      = {1976},
  publisher = {Springer}
}

@article{diamond2016cvxpy,
  author  = {Steven Diamond and Stephen Boyd},
  title   = {{CVXPY}: {A} {P}ython-embedded modeling language for convex optimization},
  journal = {Journal of Machine Learning Research},
  year    = {2016},
  volume  = {17},
  number  = {83},
  pages   = {1--5},
}

@article{dirkse1995path,
  title     = {The path solver: a nommonotone stabilization scheme for mixed complementarity problems},
  author    = {Dirkse, Steven P. and Ferris, Michael C.},
  journal   = {Optimization methods and software},
  volume    = {5},
  number    = {2},
  pages     = {123--156},
  year      = {1995},
  publisher = {Taylor \& Francis}
}

@article{dreves2012nonsmooth,
  title     = {Nonsmooth optimization reformulations of player convex generalized {N}ash equilibrium problems},
  author    = {Dreves, Axel and Kanzow, Christian and Stein, Oliver},
  journal   = {Journal of Global Optimization},
  volume    = {53},
  number    = {4},
  pages     = {587--614},
  year      = {2012},
  publisher = {Springer}
}

@article{dreves2013globalized,
  title={A globalized {N}ewton method for the computation of normalized {N}ash equilibria},
  author={Dreves, Axel and von Heusinger, Anna and Kanzow, Christian and Fukushima, Masao},
  journal={Journal of Global Optimization},
  volume={56},
  pages={327--340},
  year={2013},
  publisher={Springer}
}

@article{dreves2018generalized,
  title={A generalized {N}ash equilibrium approach for optimal control problems of autonomous cars},
  author={Dreves, Axel and Gerdts, Matthias},
  journal={Optimal Control Applications and Methods},
  volume={39},
  number={1},
  pages={326--342},
  year={2018},
  publisher={Wiley Online Library}
}

@article{dreves2018select,
  title     = {How to select a solution in generalized {N}ash equilibrium problems},
  author    = {Dreves, Axel},
  journal   = {Journal of Optimization Theory and Applications},
  volume    = {178},
  number    = {3},
  pages     = {973--997},
  year      = {2018},
  publisher = {Springer US}
}

@article{dreves2019algorithm,
  title={An algorithm for equilibrium selection in generalized {N}ash equilibrium problems},
  author={Dreves, Axel},
  journal={Computational Optimization and Applications},
  volume={73},
  number={3},
  pages={821--837},
  year={2019},
  publisher={Springer}
}

@article{facchinei2010penalty,
  title={Penalty methods for the solution of generalized {N}ash equilibrium problems},
  author={Facchinei, Francisco and Kanzow, Christian},
  journal={SIAM Journal on Optimization},
  volume={20},
  number={5},
  pages={2228--2253},
  year={2010},
  publisher={Society for Industrial and Applied Mathematics}
}

@article{facchinei2010generalized,
  title     = {Generalized {N}ash equilibrium problems},
  author    = {Facchinei, Francisco and Kanzow, Christian},
  journal   = {Annals of Operations Research},
  volume    = {175},
  number    = {1},
  pages     = {177--211},
  year      = {2010},
  publisher = {Springer}
}

@techreport{flam1994noncooperative,
  title        = {Noncooperative convex games: computing equilibrium by partial regularization},
  author       = {Flam, S.D. and Ruszczynski, Andrzej},
  year         = {1994},
  institution  = {International Institute for Applied Systems Analysis},
  reportnumber = {WP-94-042},
  url          = {http://pure.iiasa.ac.at/id/eprint/4167/}
}

@article{fullerEA17,
  title     = {Alternative models for markets with nonconvexities},
  author    = {Fuller, J. David and {\c{C}}elebi, Emre},
  journal   = {European Journal of Operational Research},
  volume    = {261},
  number    = {2},
  pages     = {436--449},
  year      = {2017},
  publisher = {Elsevier}
}

@article{ghosh2015normalized,
  title={Normalized {N}ash equilibrium for power allocation in cognitive radio networks},
  author={Ghosh, Arnob and Cottatellucci, Laura and Altman, Eitan},
  journal={IEEE Transactions on Cognitive Communications and Networking},
  volume={1},
  number={1},
  pages={86--99},
  year={2015},
  publisher={IEEE}
}

@article{guerravazquezEA08,
  title     = {Generalized semi-infinite programming: a tutorial},
  author    = {Guerra-V{\'a}zquez, Francisco and R{\"u}ckmann, J-J and Stein, Oliver and Still, Georg},
  journal   = {Journal of Computational and Applied Mathematics},
  volume    = {217},
  number    = {2},
  pages     = {394--419},
  year      = {2008},
  publisher = {North-Holland}
}

@article{harker1991generalized,
  title     = {Generalized {N}ash games and quasi-variational inequalities},
  author    = {Harker, Patrick T.},
  journal   = {European Journal of Operational Research},
  volume    = {54},
  number    = {1},
  pages     = {81--94},
  year      = {1991},
  publisher = {Elsevier}
}

@article{harks2024generalized,
  title     = {Generalized {N}ash equilibrium problems with mixed-integer variables},
  author    = {Harks, Tobias and Schwarz, Julian},
  journal   = {Mathematical Programming},
  pages     = {1--47},
  year      = {2024},
  publisher = {Springer Berlin Heidelberg Berlin/Heidelberg}
}

@article{harwood2022equilibrium,
  title   = {Equilibrium modeling and solution approaches inspired by nonconvex bilevel programming},
  author  = {Harwood, Stuart and Trespalacios, Francisco and Papageorgiou, Dimitri and Furman, Kevin},
  journal = {arXiv preprint arXiv:2107.01286v2},
  year    = {2022}
}

@article{harwood2024equilibrium,
  title={Equilibrium modeling and solution approaches inspired by nonconvex bilevel programming},
  author={Harwood, Stuart and Trespalacios, Francisco and Papageorgiou, Dimitri and Furman, Kevin},
  journal={Computational Optimization and Applications},
  volume={87},
  number={2},
  pages={641--676},
  year={2024},
  publisher={Springer}
}

@article{krawczyk05,
  title     = {Coupled constraint {N}ash equilibria in environmental games},
  author    = {Krawczyk, Jacek B.},
  journal   = {Resource and Energy Economics},
  volume    = {27},
  number    = {2},
  pages     = {157--181},
  year      = {2005},
  publisher = {Elsevier}
}

@article{le2022parametrized,
  title={Parametrized inexact-admm based coordination games: A normalized {N}ash equilibrium approach},
  author={Le Cadre, H{\'e}l{\`e}ne and Mou, Yuting and H{\"o}schle, Hanspeter},
  journal={European Journal of Operational Research},
  volume={296},
  number={2},
  pages={696--716},
  year={2022},
  publisher={Elsevier}
}

@article{liang2019generalized,
  title={A generalized {N}ash equilibrium approach for autonomous energy management of residential energy hubs},
  author={Liang, Yile and Wei, Wei and Wang, Cheng},
  journal={IEEE Transactions on Industrial Informatics},
  volume={15},
  number={11},
  pages={5892--5905},
  year={2019},
  publisher={IEEE}
}

@incollection{MaherMiltenbergerPedrosoRehfeldtSchwarzSerrano2016,
  author = {Stephen Maher and Matthias Miltenberger and Jo{\~{a}}o Pedro Pedroso and Daniel Rehfeldt and Robert Schwarz and Felipe Serrano},
  title = {{PySCIPOpt}: Mathematical Programming in Python with the {SCIP} Optimization Suite},
  booktitle = {Mathematical Software {\textendash} {ICMS} 2016},
  publisher = {Springer International Publishing},
  pages = {301--307},
  year = {2016},
  doi = {10.1007/978-3-319-42432-3_37},
}

@article{papageorgiou2023pooling,
  title={Pooling problems under perfect and imperfect competition},
  author={Papageorgiou, Dimitri J. and Harwood, Stuart M. and Trespalacios, Francisco},
  journal={Computers \& Chemical Engineering},
  volume={169},
  pages={108067},
  year={2023},
  publisher={Elsevier}
}

@article{rosen1965existence,
  title={Existence and uniqueness of equilibrium points for concave n-person games},
  author={Rosen, J. Ben},
  journal={Econometrica: Journal of the Econometric Society},
  pages={520--534},
  year={1965},
  publisher={JSTOR}
}

@article{sagratella2019generalized,
  title     = {On generalized {N}ash equilibrium problems with linear coupling constraints and mixed-integer variables},
  author    = {Sagratella, Simone},
  journal   = {Optimization},
  volume    = {68},
  number    = {1},
  pages     = {197--226},
  year      = {2019},
  publisher = {Taylor \& Francis}
}

@article{vonheusingerEA09,
  title     = {Optimization reformulations of the generalized {N}ash equilibrium problem using {N}ikaido-{I}soda-type functions},
  author    = {von Heusinger, Anna and Kanzow, Christian},
  journal   = {Computational Optimization and Applications},
  volume    = {43},
  number    = {3},
  pages     = {353--377},
  year      = {2009},
  publisher = {Springer}
}

@article{vonheusingerEA12,
  title     = {Newton's method for computing a normalized equilibrium in the generalized {N}ash game through fixed point formulation},
  author    = {von Heusinger, Anna and Kanzow, Christian and Fukushima, Masao},
  journal   = {Mathematical Programming},
  volume    = {132},
  number    = {1},
  pages     = {99--123},
  year      = {2012},
  publisher = {Springer-Verlag}
}

@article{wang2019generalized,
  title={A generalized {N}ash equilibrium game model for removing regional air pollutant},
  author={Wang, Qin and Zhao, Laijun and Guo, Lei and Jiang, Ran and Zeng, Lijun and Xie, Yujing and Bo, Xin},
  journal={Journal of cleaner production},
  volume={227},
  pages={522--531},
  year={2019},
  publisher={Elsevier}
}

@article{wei2015charging,
  title={Charging strategies of {EV} aggregator under renewable generation and congestion: A normalized {N}ash equilibrium approach},
  author={Wei, Wei and Liu, Feng and Mei, Shengwei},
  journal={IEEE Transactions on Smart Grid},
  volume={7},
  number={3},
  pages={1630--1641},
  year={2015},
  publisher={IEEE}
}

@article{wu2020energy,
  title={Energy trading and generalized {N}ash equilibrium in combined heat and power market},
  author={Wu, Chenyu and Gu, Wei and Bo, Rui and MehdipourPicha, Hossein and Jiang, Ping and Wu, Zhi and Lu, Shuai and Yao, Shuai},
  journal={IEEE Transactions on Power Systems},
  volume={35},
  number={5},
  pages={3378--3387},
  year={2020},
  publisher={IEEE}
}

@article{zhou2005generalized,
  title={The generalized {N}ash equilibrium model for oligopolistic transit market with elastic demand},
  author={Zhou, Jing and Lam, William HK and Heydecker, Benjamin G},
  journal={Transportation Research Part B: Methodological},
  volume={39},
  number={6},
  pages={519--544},
  year={2005},
  publisher={Elsevier}
}
\end{document}